\documentclass[a4paper,oneside,12pt]{amsart}
\usepackage[english]{babel}
\usepackage{amsfonts,amsmath,amsthm,amssymb,latexsym,mathrsfs}

\newtheorem{theorem}{Theorem}[section]
\newtheorem{corollary}[theorem]{Corollary}
\newtheorem{lemma}[theorem]{Lemma}
\newtheorem{proposition}[theorem]{Proposition}

\theoremstyle{plain}

\newtheorem{remark}[theorem]{Remark}

\begin{document}

\title[Extension of bounded Baire-one functions]{Extension of bounded Baire-one functions \\vs\\ Extension of unbounded Baire-one functions}
\author{Olena Karlova${}^1$}

\address{${}^1$Department of Mathematics and Informatics, Chernivtsi National University, Ukraine}
\email{maslenizza.ua@gmail.com}

\author{Volodymyr Mykhaylyuk${}^{1,2}$}
\address{${}^2$Jan Kochanowski University in Kielce, Poland}
\email{vmykhaylyuk@ukr.net}


\begin{abstract}
  We compare possibilities of extension of bounded and unbounded Baire-functions from subspaces of topological spaces.
\end{abstract}
\maketitle

\section{Introduction}

Let $X$   be a topological space. A function $f:X\to \mathbb R$ belongs \emph{to the first Baire class}, if it is a pointwise limit of a sequence of real-valued continuous functions on $X$. We will denote by ${\rm B}_1(X)$ and ${\rm B}_1^*(X)$ the collections of all Baire-one and bounded Baire-one functions on $X$, respectively.

A subset $E$ of $X$ is \emph{${\rm B}_1$-embedded} (\emph{${\rm B}_1^*$-embedded}) in $X$, if every (bounded) function  $f\in {\rm B}_1(E)$ can be extended to a function $g\in {\rm B}_1(X)$. We will say that a space $X$ has the \emph{property} (${\rm B}_1^*={\rm B}_1$) if every ${\rm B}_1^*$-embedded subset of $X$ is ${\rm B}_1$-embedded in $X$.

Characterizations of ${\rm B}_1$- and ${\rm B}_1^*$-embedded subsets of topological spaces were obtained in~\cite{Ka} and \cite{Karlova:CMUC:2013}.

This short note is devoted to interesting problem: to find topological spaces with the property (${\rm B}_1^*={\rm B}_1$).

In the second section of this note we extend results from \cite[Section 6]{Karlova:CMUC:2013}  and show that  every hereditarily Lindel\"{o}ff hereditarily Baire space $X$ which hereditarily has a $\sigma$-discrete $\pi$-base has the property (${\rm B}_1^*={\rm B}_1$). In Section~\ref{B1neB1star} we show that any countable completely regular  hereditarily irresolvable space $X$ without isolated points is ${\rm B}_1^*$-embedded and is not ${\rm B}_1$-embedded in $\beta X$.

\section{Spaces with the property (${\rm B}_1^*={\rm B}_1$)}

Recall that a set $A$ in a topological space $X$ is \emph{functionally $G_\delta$ (functionally $F_\sigma$)}, if $A$ is an intersection (a union) of  a sequence of functionally open (functionally closed) subsets of $X$. We say that a subsets $A$ of a topological space $X$ is \emph{functionally ambiguous} if $A$ is functionally $F_\sigma$ and functionally $G_\delta$ simultaneously.

\begin{lemma}\label{lem:resolvability-Borel}
  Let $X$ be a completely regular topological space of the first category with a $\sigma$-discrete $\pi$-base. Then there exist  disjoint functionally ambiguous sets $A$ and $B$ such that
  $$
  X=A\cup B=\overline{A}=\overline{B}.
  $$
\end{lemma}

\begin{proof}
  We fix  a $\pi$-base $\mathscr V=(\mathscr V_n:n\in\omega)$ of $X$, where  each family $\mathscr V_n$ is discrete and consists of functionally open sets in $X$. Denote $V_n=\bigcup\{V: V\in\mathscr V_n\}$ for all $n\in\omega$.

Let us observe that every open set $G\subseteq X$ contains functionally open subset  $U$ such that $U\subseteq G\subseteq \overline{U}$. Indeed, for every $n\in\omega$ we put $U_n=\cup\{V\in\mathscr V_n: V\subseteq G\}$ and $U=\bigcup_{n\in\omega} U_n$. Then each $U_n$ is functionally open as a union of a discrete family of functionally open sets. Hence, $U$ is functionally open. It is easy to see that $U$ is dense in $G$.

Keeping in mind the previous fact, we may assume that  there exists a covering $(F_n:n\in\omega)$ of the space $X$ by nowhere dense functionally closed sets $F_n\subseteq X$. Let $X_0=F_0$ and $X_n=F_n\setminus \bigcup_{k<n}F_k$ for all $n\ge 1$. Then $(X_n:n\in\omega)$ is a partition of $X$ by nowhere dense functionally ambiguous sets $X_n$.

Fix $n\in\omega$ and $V\in\mathscr V_n$. Since $X$ is regular, we can choose two open sets $H_1$ and $H_2$ in $V$ such that $\overline{H_1}\cap\overline{H_2}=\emptyset$ and $\overline{H_i}\subseteq V$ for $i=1,2$. Let $G_i$ and $O_i$ be functionally open sets such that $G_i\subseteq H_i\subseteq\overline{G_i}$ and $O_i\subseteq X\setminus \overline{H_i}\subseteq \overline{O_i}$, $i=1,2$. We put $A_{V,n}=X\setminus (G_1\cup O_1)$ and $B_{V,n}=X\setminus (G_2\cup O_2)$ and obtain disjoint nowhere dense functionally closed subsets of $V$.

We put $m_0=0$ and choose numbers $n_1\ge 0$ and $m_1>n_1$ such that $X_{n_1}\cap V_1\ne\emptyset$ and $X_{m_1}\cap V_1\ne\emptyset$. Notice that
$A_1'=\bigcup_{n=0}^{n_1}X_n$ and $B_1'=\bigcup_{n=n_1+1}^{m_1}X_n$ are nowhere dense functionally ambiguous sets in $X$.
Now we consider a set $\mathscr W_1=\{V\in\mathscr V_1: V\cap (A_1'\cup B_1')=\emptyset\}$ and observe that the sets
$A_1''=\cup\{A_{V,1}:V\in\mathscr W_1\}$ and $B_1''=\cup\{B_{V,1}:V\in\mathscr W_1\}$ are functionally closed and nowhere dense in $X$. Let
$A_1=A_1'\cup A_1''$ and $B_1=B_1'\cup B_1''$.
Notice that $A_1$ and $B_1$ are functionally ambiguous nowhere dense disjoint subsets of $X$.

Since $\overline{X\setminus (A_1\cup B_1)}=X$, there exists a number $n_2>m_1$ such that $(X_{n_2}\setminus (A_1\cup B_1))\cap V_2\ne\emptyset$. We put
$A_2'=\bigcup_{n=m_1+1}^{n_2} (X_n\setminus (A_1\cup B_1))$. Moreover, there exists $m_2>n_2$ such that $(X_{m_2}\setminus (A_1\cup B_1))\cap V_2\ne\emptyset$. Let $B_2'=\bigcup_{n=n_2+1}^{m_2} (X_n\setminus (A_1\cup B_1))$. We put $\mathscr W_2=\{V\in\mathscr V_2: V\cap (A_2'\cup B_2')=\emptyset\}$ and observe that the sets $A_2''=\{A_{V,2}:V\in\mathscr W_2\}$ and $B_2''=\{b_{V,2}:V\in\mathscr W_2\}$ are functionally closed and nowhere dense in $X$. We denote
$A_2=A_2'\cup A_2''$ and $B_2=B_2'\cup B_2''$. Then $A_2$ and $B_2$ are functionally ambiguous nowhere dense disjoint subsets of $X$.

Proceeding this process inductively we obtain  sequences $(A_k)_{k=1}^\infty$ and $(B_k)_{k=1}^\infty$ of functionally ambiguous sets such that $A_k\cap V\ne\emptyset\ne B_k\cap V$, $A_k\cap B_k=\emptyset$  or all $k\in\mathbb N$ and $V\in \mathscr V_k$. It remains to put $A=\bigcup_{k=1}^\infty A_k$, $B=\bigcup_{k=1}^\infty B_k$ and observe that $A\cup B=X$.
\end{proof}

In addition, note that Borel resolvability of topological spaces was also studied in~\cite{Ced, JM}.

We say that a topological space $X$ \emph{hereditarily has a $\sigma$-discrete $\pi$-base} if every its closed subspace has a $\sigma$-discrete $\pi$-base. It is easy to see that if a space $X$ hereditarily has a $\sigma$-discrete $\pi$-base, then each subspace of $X$ has a $\sigma$-discrete $\pi$-base.

\begin{lemma}\label{lem:hB-sub}
  Let $X$ be a normal  space such that $X$ hereditarily has a $\sigma$-discrete $\pi$-base.  If $X$ is a  ${\rm B}_1^*$-embedded subset of a hereditarily Baire space $Y$, then $X$ is hereditarily Baire.
\end{lemma}

\begin{proof}
  Assume that $X$ is not hereditarily Baire and find a closed subset $F\subseteq X$ of the first category. According to Lemma~\ref{lem:resolvability-Borel}, there exist disjoint functionally ambiguous subsets $A$ and $B$ in $F$ such that $F=A\cup B=\overline{A}=\overline{B}$. Since $F$ is a closed subset of a normal space, $F$ is $z$-embedded in $X$. Therefore, there are two functionally ambiguous disjoint sets $\widetilde A$ and $\widetilde B$ in $X$ such that $\widetilde A\cap F=A$ and $\widetilde B\cap F=B$ (see \cite[Proposition 4.3]{Karlova:CMUC:2013}). Let us observe that   the characteristic function $\chi:X\to [0,1]$ of the set $\widetilde A$ belongs to the first Baire class. Then there exists an extension $f\in{\rm B}_1(Y)$ of $\chi$. The sets $f^{-1}(0)$ and $f^{-1}(1)$ are disjoint $G_\delta$-sets which are dense in $\overline{X}$. We obtain a contradiction, because $\overline{X}$ is a Baire space as a closed subset of a hereditarily Baire space.
\end{proof}

\begin{remark}
  {\rm {\it There exist a metrizable separable Baire space $X$ and its ${\rm B}_1^*$-embedded subspace $E$ which is not a Baire space.} Let  $X=({\mathbb Q}\times \{0\})\cup ({\mathbb R}\times (0,1])$ and  $E={\mathbb Q}\times \{0\}$. Then $E$ is closed in $X$. Therefore, any $F_\sigma$- and $G_\delta$-subset $C$ of $E$ is also $F_\sigma$- and $G_\delta$- in $X$. Hence, $E$ is ${\rm B}_1^*$-embedded  in $X$.}
\end{remark}

\begin{theorem}\label{B1=B1star}
 Let $Y$ be a hereditarily Baire completely regular space and $X\subseteq Y$ be a Lindel\"{o}ff space which hereditarily has a $\sigma$-discrete $\pi$-base. The following are equal:
 \begin{enumerate}
   \item $X$ is ${\rm B}_1^*$-embedded in $Y$;

   \item $X$ is ${\rm B}_1$-embedded in $Y$.
 \end{enumerate}
\end{theorem}

\begin{proof}
  We need only  to show 1) $\Rightarrow$ 2). By Lemma~\ref{lem:hB-sub}, $X$ is hereditarily Baire. Then $X$ is ${\rm B}_1$-embedded in $Y$ by \cite[Theorem 13]{Ka}.
\end{proof}

\begin{corollary}
  Every hereditarily Lindel\"{o}ff hereditarily Baire space $X$ which hereditarily has a $\sigma$-discrete $\pi$-base has the property (${\rm B}_1^*={\rm B}_1$).
\end{corollary}

\section{Spaces without the property (${\rm B}_1^*={\rm B}_1$)}\label{B1neB1star}

  A subset $A$ of a topological space $X$ is called {\it (functionally) resolvable in the sense of Hausdorff} or {\it (functionally) $H$-set} if
  \begin{gather*}\label{gath:resolvable_set}
  A=(F_1\setminus F_2)\cup (F_3\setminus F_4)\cup\dots\cup(F_\xi\setminus F_{\xi+1})\cup\dots,
  \end{gather*}
where  $(F_\xi)_{\xi<\alpha}$ is a decreasing chain of (functionally) closed sets in $X$.

It is well-known~\cite[\S 12.I]{Kur:Top:Rus1} that a set $A$ is an $H$-set if and only if for any closed nonempty set   $F\subseteq X$ there is a nonempty relatively open set  $U\subseteq F$ such that $U\subseteq A$ or $U\subseteq X\setminus A$.


A topological space without isolated points is called \emph{crowded}.

A topological space $X$ is \emph{irresolvable} if is not a union of two disjoint dense subsets and is \emph{hereditarily irresolvable} if every subspace of $X$ is irresolvable.

\begin{lemma}\label{prop:resofirres}
 Every subset of a hereditarily irresolvable space is an H-set.
\end{lemma}

\begin{proof} Assume that there is a closed nonempty set $F$ in a hereditarily irresolvable space $X$ and a set $A\subseteq X$ such that
 $\overline{F\cap A}\cap\overline{F\setminus A}=F$. Then $\overline{F\cap A}=\overline{F\setminus A}=F=(F\cap A)\cup (F\setminus A)$, which contradicts to irresolvability of  $F$.
\end{proof}

A function $f:X\to Y$ between a topological space $X$ and a metric space $(Y,d)$ is called \emph{fragmented} if for every $\varepsilon>0$ and for every closed nonempty set $F\subseteq X$ there exists a relatively open nonempty set $U\subseteq F$ such that ${\rm diam}f(U)<\varepsilon$.

\begin{proposition}\label{prop:boundfragm}
 Every bounded function $f:X\to \mathbb R$ on a hereditarily irresolvable space $X$ is fragmented.
\end{proposition}

\begin{proof}
  To obtain a contradiction we assume that there exists a bounded function $f:X\to\mathbb R$ which is not fragmented. Then there is $\varepsilon>0$ and a closed nonempty set $F\subseteq X$ such that for every relatively open set $U\subseteq F$ we have ${\rm diam} f(U)\ge\varepsilon$.
  
  Since $f(X)$ is a compact set, we take a finite partition $\{B_1,\dots,B_n\}$ of $f(X)$ by sets with diameters $<\varepsilon$. Let $H_k=f^{-1}(B_k)\cap F$ for every $k\in\{1,\dots,n\}$. Then each $H_k$ have empty interior in $F$, because $f$ is not fragmented. By Lemma~\ref{prop:resofirres}, each $H_k$ is an $H$-set and, therefore, is nowhere dense in $F$. Hence, $\{H_1,\dots,H_n\}$ is a finite partition of $F$ by nowhere dense sets, which is impossible. 
\end{proof}

Recall that a subspace $E$ of a topological space $X$ is \emph{$z$-embedded in $X$}, if any functionally closed subset $F$ of $E$ can be extended to a functionally closed subset of $X$.

\begin{lemma}\label{prop:ext:restofsigma}
  Let $E$ be a $z$-embedded countable subspace of a topological space  $X$ and $A\subseteq E$ be a functionally H-set in  $E$. Then there exists a functionally H-set  $B\subseteq X$ such that $B$ is $F_\sigma$ and $B\cap E=A$.
\end{lemma}

\begin{proof}
 We take a decreasing transfinite sequence  $(A_\xi:\xi<\alpha)$ of functionally closed subsets of   $E$ such that
  $A=\bigcup_{\xi<\alpha} (A_\xi\setminus A_{\xi+1})$ (every ordinal $\xi$ is odd). Since $|A|\le\aleph_0$, we may assume that   $|(A_\xi:\xi<\alpha)|\le \aleph_0$. The subspace $E$ is   $z$-embedded in $X$ and we choose a decreasing sequence  $(B_\xi:\xi<\alpha)$ of functionally closed sets in   $X$ such that $A_\xi=B_\xi\cap E$ for all $\xi<\alpha$. We put
  $$
  B=\bigcup_{\xi<\alpha,\,\xi {\footnotesize \mbox{\,\,is odd}}}(B_\xi\setminus B_{\xi+1}).
  $$
  Then $B$ is functionally   $F_\sigma$-set in  $X$ and $B\cap E=A$.
\end{proof}

\begin{lemma}\label{prop:resincomp}
 Let $X$ be a compact space and $B\subseteq X$ be functionally Borel measurable H-set. Then  $B$ is functionally ambiguous in $X$.
\end{lemma}

\begin{proof}
  Since $B$ is functionally Borel measurable, there exists a sequence  $(f_n)_{n\in\omega}$ of continuous functions $f_n:X\to[0,1]$ such that   $B$ belongs to the  $\sigma$-algebra generated by the system of sets  $\{f_n^{-1}(0):n\in\omega\}$. We consider a continuous map   $f:X\to [0,1]^\omega$, $f(x)=(f_n(x))_{n\in\omega}$ for all $x\in X$, and a compact metrizable space   $Y=f(X)\subseteq [0,1]^\omega$.

We show that the set  $B'=f(B)$ is an H-set in  $Y$. Suppose to the contrary that there is a closed nonempty set   $Y'$ in $Y$ such that  $\overline{Y'\cap B'}=\overline{Y'\setminus B'}=Y'$. We put   $X'=f^{-1}(Y')$ and $g=f|_{X'}$. Since $X'$ is a compact space and $f(X')=Y'$, we apply Zorn's Lemma and find a closed nonempty set  $Z\subseteq X'$ such that the restriction   $g|_Z:Z\to Y'$ of the continuous map $g:X'\to Y'$ is irreducible.
Keeping in mind  that the preimage of any everywhere dense set remains everywhere dense under an irreducible map, we obtain that
$$
\overline{g^{-1}(Y'\cap B')}=\overline{g^{-1}(Y'\setminus B')}=Z=\overline{Z\cap B}=\overline{Z\setminus B},
$$
which contradicts to resolvability of  $B$.

By \cite[\S 30, X, Theorem 5]{Kur:Top:Rus1} the set   $f(B)$ is $F_\sigma$ and $G_\delta$ in a compact metrizable space  $Y$. Since  $B=f^{-1}(f(B))$ and $f$ is continuous, we have that   $B$ is functionally ambiguous subset of   $X$.
\end{proof}

\begin{proposition}\label{prop:b1star}
  Let $X$ be a countable hereditarily irresolvable completely regular space. Then   $X$ is ${\rm B}_1^*$-embedded in $\beta X$.
\end{proposition}

\begin{proof}
Since  $X$ is countable and completely regular, it is perfectly normal. Therefore, every subsets of $X$ is functionally ambiguous.

Fix an arbitrary $A\subseteq X$. By Lemma~\ref{prop:resofirres} the set  $A$ is an $H$-set. We apply Lemma~\ref{prop:ext:restofsigma} and fine a functionally $H$-set $B\subseteq\beta X$ such that $B$ is   $F_\sigma$ and $B\cap X=A$. Notice that $B$ is functionally ambiguous by Lemma~\ref{prop:resincomp}. Hence, $B$ is a ${\rm B}_1^*$-embedded subspace of $\beta X$ according to~\cite[Proposition 5.1]{Karlova:CMUC:2013}.
\end{proof}

Let us observe that examples of countable hereditarily irresolvable completely regular spaces can be found, for instance, in~\cite[p. 536]{Tk}.

\begin{proposition}\label{prop:notb1}
  Let $X$ be a countable completely regular space without isolated points. Then  $X$  is not  ${\rm B}_1$-embedded in  $\beta X$.
\end{proposition}

\begin{proof} Observe that  $X$ is functionally   $F_\sigma$-subset of   $\beta X$. Now assume that  $X$ is    ${\rm B}_1$-embedded in $\beta X$.
 According to \cite[Proposition 8(iii)]{Ka} there should be a function  $f\in{\rm B}_1(\beta X)$ such that  $X\subseteq f^{-1}(0)$ and $\beta X\setminus X\subseteq f^{-1}(1)$. Then the set $X$ is  $G_\delta$ in  $\beta X$. Therefore, $X$ is a Baire space, which implies a contradiction, since $X$ is of the first category in itself.
\end{proof}

Propositions~\ref{prop:b1star} and \ref{prop:notb1} imply the following fact.

\begin{theorem}
  Let $X$ be a countable hereditarily irresolvable completely regular space without isolated points. Then $X$ is ${\rm B}_1^*$-embedded in $\beta X$ and is not ${\rm B}_1$-embedded in $\beta X$.
\end{theorem}

\section{Acknowledgements}

The first author is partially supported by a grant of Chernivtsi National University for young scientists in 2018 year.

\end{document}